\documentclass[11pt]{article}  

\usepackage{graphicx}

\usepackage{amssymb}
\usepackage{latexsym}
\usepackage[english]{babel}
\usepackage{amsmath}
\usepackage{makeidx}
\usepackage[utf8]{inputenc}
\usepackage{verbatim}
\usepackage[pagewise]{lineno}
\usepackage{amsthm}
\usepackage{tikz-cd} 

\newtheorem{theorem}{Theorem}[section]
\newtheorem{corollary}[theorem]{Corollary}

\newtheorem{lemma}[theorem]{Lemma}
\newtheorem{proposition}[theorem]{Proposition}

\theoremstyle{definition}
\newtheorem{definition}[theorem]{Definition}
\newtheorem{remark}[theorem]{Remark}
\newtheorem{example}{Example}

\newcommand{\R}{\ensuremath{\mathbb{R}}}

\newcommand{\SP}{\ensuremath{\mathcal{P}}}

\newcommand{\Hcal}{\ensuremath{\mathcal{H}}}

\newcommand{\D}{\ensuremath{\mathcal{D}}}
\newcommand{\Es}{\ensuremath{\mathbb{S}}}

\newcommand{\Reeb}{\ensuremath{\mathcal{R}}}

\def\david{\textcolor{black}}

\usepackage[style=alphabetic, maxbibnames=7, backend=bibtex, doi = false, isbn=false, url=false]{biblatex}
\title{
	{Mechanical Hamiltonization of unreduced $\phi$-simple Chaplygin systems}\\
}
\date{}
\author{Alexandre Anahory Simoes\thanks{School of Science and Technology, IE University, Madrid, Spain. ({alexandre.anahory@ie.edu}).} \and Juan Carlos Marrero\thanks{ULL-CSIC Geometría Diferencial y Mecánica Geométrica, Departamento de Matemáticas, Estadística e Investigación Operativa and Instituto de Matemáticas y Aplicaciones (IMAULL), Faculty of Sciences, University of La Laguna, Canary Islands, Spain. ({jcmarrer@ull.edu.es}).} \and David Mart\'in de Diego\thanks{Instituto de Ciencias Matematicas (CSIC), Calle Nicol\'as Cabrera 13-15, Cantoblanco, 28048, Madrid, Spain ({david.martin@icmat.es}).}}

\begin{document}
	
	\maketitle
	
	\begin{abstract}
		In this paper, we prove that the trajectories of unreduced $\phi$-simple Chaplygin kinetic systems are reparametrizations of horizontal geodesics with respect to a modified Riemannian metric. Furthermore, our proof is constructive and these Riemannian metrics, which are not unique, are obtained explicitly in interesting examples. We also extend these results to $\phi$-simple Chaplygin mechanical systems (not necessarily kinetic).
	\end{abstract}


	
	\section{Introduction}

    \subsection{Hamiltonization and invariant volume forms of reduced $\phi$-simple Chaplygin systems and our problem}

    A classical theory in nonholonomic mechanics is the so-called Hamiltonization problem (cf. \cite{Chaplygin2008, Fedorov2004, Ehlers2005, Bolsinov2011}). This problem, which has received a lot of attention in recent years (see \cite{BaYa} and the references therein; see also \cite{GaMa}), consists in arguing when a symmetric nonholonomic mechanical system admits a Hamiltonian formulation, up to reparametrization, after reduction by symmetries.
    
    $\phi$-simple Chaplygin systems are Hamiltonizable in the previous sense. We recall that a Chaplygin system is a nonholonomic mechanical system whose constraint distribution is the horizontal subbundle associated with a principal connection in a principal $G$-bundle and, in addition, the Lie group $G$ is a symmetry for the system (see \cite{koiller92}; see also \cite{Bloch, Cortes}). A tensor field $\mathcal{T}$ of type $(1, 2)$ on the base space $\bar{Q}=Q/G$ of the principal bundle, called the gyroscopic tensor, plays an important role in the description of the geometry of the reduced nonholonomic system. \david{ $\mathcal{T}$ measures the interplay between the kinetic Riemannian metric and the non-integrability of the constraint distribution}. In fact, using $\mathcal{T}$, one may construct an almost symplectic structure $\Omega_{nh}$ on the reduced phase space $T^{*}(Q/G)$ such that the reduced nonholonomic dynamics is Hamiltonian with respect to $\Omega_{nh}$. If the gyroscopic tensor $\mathcal{T}$ satisfies 
    \begin{equation}\label{def:gyroscopic}
        \mathcal{T} = Id \otimes d\phi - d\phi \otimes Id
    \end{equation}
    with $\phi\in C^{\infty}(\bar{Q})$, the Chaplygin system is $\phi$-simple \cite{GaMa}.
    
    A local characterization of $\phi$-simple Chaplygin systems and non-trivial examples of such systems were obtained in \cite{Garcia-Naranjo2019a} (see also \cite{Garcia-Naranjo2019b, GaMa}). Moreover, in \cite{GaMa}, it was proved that a Chaplygin system is $\phi$-simple if and only if the almost symplectic structure $\Omega_{nh}$ is conformally symplectic (we stress that the proof of the direct implication of this statement is a consequence of a result by Stanchenko \cite{Stanchenko1989}). So, as we mentioned before, $\phi$-simple Chaplygin systems are Hamiltonizable. In fact, the conformal factor for $\Omega_{nh}$ is related with the invariant volume form for the reduction of the Chaplygin system (for more details see \cite{GaMa}). Thus, a necessary condition for a Chaplygin system to be $\phi$-simple is the existence of an invariant volume form for the reduced system. In this direction, we remark that invariant volume forms for the reduction of Chaplygin systems were discussed in \cite{Cantrijn2002} and for general nonholonomic mechanical systems in \cite{FGNM2015} (see also \cite{Grab12}). In addition, very recently, $\phi$-simple Chaplygin systems with gyroscopic terms were considered in \cite{Dragovic2023}.
    
    
    Now, suppose that a  nonholonomic mechanical system \david{with symmetry} is Hamiltonizable. Then, a natural question arises: \textit{does the unreduced system also admit a Hamiltonian formulation?} Since we deal exclusively with nonholonomic mechanical systems, the previous question may be even more specialized: \textit{are there a (modified) Riemannian metric and a (modified) potential energy such that the trajectories of the new unconstrained mechanical system, with initial velocity in the constraint distribution, are reparametrizations of the nonholonomic trajectories of the unreduced nonholonomic mechanical system?}
    
    This question is a particular case of a more general problem, which was posed in \cite{AAMkinetic}: the kinetic Lagrangianization of kinetic nonholonomic systems or, more generally, the formulation of a nonholonomic mechanical system as an unconstrained Lagrangian mechanical system (see items 3 and 7 of Section 6 in \cite{AAMkinetic}).

    \subsection{Aim of the paper, related previous contributions and our strategy}
    
    The aim of our paper is to give a positive answer to the questions in the previous section for $\phi$-simple Chaplygin systems.
    
    We remark that, comparatively speaking, the Hamiltonization of unreduced nonholonomic mechanical systems has received much less attention than the Hamiltonization problem for the reduction of symmetric nonholonomic mechanical systems. In fact, to our knowledge, only a few papers (see \cite{BLOCH2009225, belrhazi2024geodesic} and also \cite{marta}) have discussed this problem.

    In these papers, and particularly in \cite{belrhazi2024geodesic}, the idea is to use the inverse problem \david{of variational calculus} to provide a set of conditions for the existence of a Riemannian metric, the geodesic extension in the terminology of \cite{belrhazi2024geodesic}, whose geodesics with initial velocity in the constraint distribution are the trajectories of the kinetic nonholonomic system.
    
    Our approach to the problem is different. In fact, we will use the symmetries of the $\phi$-simple Chaplygin kinetic system and the Riemannian submersion theory (\cite{ON66, ON67, ON83}) to obtain a modified Riemannian metric such that the nonholonomic trajectories are reparametrizations of geodesics of the new metric. Our method is constructive and we can give an explicit expression of the new metric. Moreover, using the previous construction for $\phi$-simple Chaplygin kinetic systems, we obtain the corresponding results for the more general case of a $\phi$-simple Chaplygin mechanical system (\david{adding a potential energy}).

    In the particular case of a $\phi$-simple Chaplygin kinetic system, the Riemannian submersion theory is used as follows. For a Riemannian submersion, the horizontal distribution $\text{Hor}$ in the total space is just the orthogonal subbundle of the vertical subbundle. $\text{Hor}$ is, in general, non-integrable and moreover, the restriction of the geodesic flow to $\text{Hor}$ is tangent. This means that if the initial velocity of a geodesic is horizontal then its velocity curve will be entirely contained in $\text{Hor}$. In addition, the projection of a horizontal geodesic is a geodesic in the base space \david{for the reduced metric} and the horizontal lift of a geodesic in the base space is a geodesic in the total space (for more details, see \cite{ON66, ON67, ON83}). Based on this fact, and for a $\phi$-simple Chaplygin kinetic system, our goal in the paper will be to reproduce this picture by constructing a new metric with respect to which the vertical space is orthogonal to the constraint distribution. In this scenario, the horizontal geodesics would also be nonholonomic trajectories. The relevant mechanism is the existence of a projection mapping nonholonomic trajectories to geodesics or, at least, onto reparametrizations of geodesics. In this case, we are able to tweak the original metric in order to make the horizontal space orthogonal to the vertical space and still preserve the Riemannian submersion structure.

    \subsection{Results of the paper}

    For a Chaplygin mechanical system, we will use the following notation $$(Q, g, V, G, \D),$$ where $Q$ is the configuration space, $g$ is \david{a $G$-invariant} Riemannian metric on $Q$, $V: Q\to \R \in C^{\infty}(Q)$ is the $G$-invariant potential energy, $G$ is the symmetry Lie group and $\D\subseteq TQ$ is the constraint distribution. When $V=0$, the system is kinetic. The base space of the principal $G$-bundle is $\bar{Q}=Q/G$ and the Riemannian metric $\bar{g}$ on $\bar{Q}$ is characterized by the condition
    $$\bar{g}(T \pi (X) , T\pi(Y)) = g(X, Y), \text{ for } X, Y \in \D,$$
    with $\pi : Q\to \bar{Q}$ the principal bundle projection.
 
	Then, the main result of the paper is the characterization of trajectories of $\phi$-simple kinetic Chaplygin systems as reparametrizations of geodesics with respect to a modified metric (Theorem \ref{extension:main:thm}).

\vspace{1mm}

 {\flushleft \bf Theorem \ref{extension:main:thm}.} {\it Let $(Q,g, G, \D)$ be a kinetic Chaplygin system and $\pi:Q\rightarrow \bar{Q}$ the projection onto the quotient space. If the system is $\phi$-simple, there exists a Riemannian metric $h$ on $Q$ such that all its geodesics starting in $\D$ are reparametrizations of the nonholonomic trajectories associated with the Chaplygin system. In fact, if $V\pi=\ker T\pi$ and $g_{can}$ is the Riemannian metric on $\bar{Q}$ defined by $g_{can}=e^{2\phi}\bar{g}$ then $h$ may be chosen as follows
        \begin{enumerate}
			\item $h(\bar{X}^{h}, \bar{Y}^{h}) = g_{can}(\bar{X}, \bar{Y}), \forall \bar{X}, \bar{Y}\in \mathfrak{X}(\bar{Q})$;
			\item $h(Z, W) = g(Z, W), \forall Z,W\in \Gamma(V\pi)$;
			\item $h(\bar{X}^{h}, Z)= 0, \forall \bar{X}\in \mathfrak{X}(\bar{Q}), \ Z\in \Gamma(V\pi)$.
		\end{enumerate}
    Here , $\bar{X}^h$, $\bar{Y}^h$ are the horizontal lifts of $\bar{X}$, $\bar{Y}$, respectively, by the principal connection associated with the Chaplygin system.}
 
    \vspace{1mm}
    
    Note that our result is constructive since it provides an explicit construction of one of these metrics, which we will call the \textit{principal Riemannian metric}. 


    An immediate consequence of our main theorem is an interesting result. Namely, the nonholonomic trajectories become locally minimizing curves with respect to the new metric, as the following corollary states.

    {\flushleft \bf Corollary \ref{Riemannian:distance:cor}.} {\it
        Let $(Q,g, G, \D)$ be a kinetic Chaplygin system. Suppose that the system is $\phi$-simple and $h$ is the principal Riemannian metric.  If $c:[0,T]\to Q$ is a nonholonomic trajectory there exists $0<t_{0}<T$ such that for all $t<t_{0}$ the length of the curve $c:[0,t]\to Q$ is just the Riemannian distance associated with $h$ between $c(0)$ and $c(t)$.}
    
    \vspace{1mm}

\david{As a demonstration of the applicability in practical examples, we have constructed explicit examples of Riemannian metrics whose geodesics with initial velocity satisfying the nonholonomic constraints coincide up to re-parametrization with the trajectories of the initial nonholonomic system.}

For instance, in the classical example of the vertical rolling disk, the principal Riemannian metric is:
\begin{equation*}
    h = m dx^{2} + m dy^{2} + (I + 2 mR^{2}) d\theta^{2} + J d\varphi^{2} - mR \cos \varphi \ dx d\theta - mR \sin\varphi \ dy d\theta.
\end{equation*}
Here, $(x,y,\theta, \varphi)$ are standard coordinates in the configuration space $Q=\R^{2}\times \Es^{1} \times \Es^{1}$, $m$ and $R$ are the mass and radius of the disk and $I, J$ are the inertia.

In fact, in Example \ref{example1}, we compute the geodesic equations for $h$ and deduce that the geodesics with initial velocity in the constraint distribution coincide with the nonholonomic trajectories, without the need of a reparametrization.

In the nonholonomic particle (see Example \ref{example2}), the configuration space is $\R^{3}$, $(x,y,z)$ are the standard coordinates on $\R^{3}$ and our method allows us to construct the principal Riemannian metric
$$h = (1+y^{2}) dx^{2} + \frac{dy^{2}}{1+y^{2}} + dz^{2} - y dx dz,$$
whose geodesics with initial velocity satisfying $\dot{z}=y\dot{x}$ are precisely the nonholonomic trajectories, up to a reparametrization. 
An other interesting example is the case of the Veselova problem where we also provide (see Example \ref{example3}) the principal Riemannian metric describing the nonholonomic trajectories.

In all three cases, we can observe that the reparametrization is related to the function $\phi$, determining the $\phi$-simple nature of the system. The vertical rolling disk is a special case where $ \phi\equiv 0$ forcing the associated reparametrization to be the identity.

We remark that the method has been applied to the three previous examples but, since it is constructive, it may be applied to more complex cases (see Remark \ref{multi:dimensional:Veselova}).

Finally, the main result of the paper is extended to more general $\phi$-simple Chaplygin systems where one has a $G$-invariant potential function in addition to the kinetic energy.

    {\flushleft \bf Theorem \ref{potential:thm}.} {\it Let $(Q,g, V, G, \D)$ be a Chaplygin system with $G$-invariant potential $V$. If the system is $\phi$-simple, there exists a Riemannian metric $h$ on $Q$ such that all the mechanical trajectories associated with $h$ and $V$ with initial velocity in $\D$ are reparametrizations of the nonholonomic trajectories associated with the Chaplygin system. In fact, $h$ may be chosen as follows
        \begin{enumerate}
			\item $h(\bar{X}^{h}, \bar{Y}^{h}) = g_{can}(\bar{X}, \bar{Y}), \forall \bar{X}, \bar{Y}\in \mathfrak{X}(\bar{Q})$;
			\item $h(Z, W) = g(Z, W), \forall Z,W\in \Gamma(V\pi)$;
			\item $h(\bar{X}^{h}, Z)= 0, \forall \bar{X}\in \mathfrak{X}(\bar{Q}), \ Z\in \Gamma(V\pi)$.
		\end{enumerate}
    Here, $V\pi$ is the vertical bundle associated to $\pi$ and $\bar{X}^h$, $\bar{Y}^h$ are the horizontal lifts of $\bar{X}$, $\bar{Y}$, respectively, by the principal connection associated with the Chaplygin system.}

    \subsection{Organization of the paper}
    
    The paper is organized as follows. In Section 2, we first briefly review kinetic nonholonomic systems, Riemannian submersions, the notion of a principal connection on a principal $G$-bundle and Chaplygin systems. This section is intended to make the paper self-contained and settle the notation that we use throughout the paper. In Section 3 we prove our main result: unreduced $\phi$-simple Chaplygin kinetic systems are Hamiltonizable, up to reparametrization, by geodesics. In Section 4, we show that the previous results are not essentially changed by the inclusion of a $G$-invariant potential function into the picture. In Section 5, we present examples of typical $\phi$-simple Chaplygin mechanical systems (the vertical rolling disk, the nonholonomic particle and the Veselova system) where we construct the principal Riemannian metric Hamiltonizing the nonholonomic trajectories. Finally, in Section 6, we present the conclusions and some lines of research for future work.

	\section{Preliminaries}
	
	\subsection{Nonholonomic systems}
	
	Throughout this section suppose that $(Q,g)$ is a Riemannian manifold, $\D$ is a distribution on $Q$ and consider the \textit{kinetic nonholonomic system} defined by the Lagrangian function $L_{g}:TQ\rightarrow \R$ given by
	$$L_{g}(v_{q}) = \frac{1}{2} g(v_{q}, v_{q})\, , \quad v_q\in T_qQ$$
	and also by \david{a nonintegrable} distribution $\D$. We may define two projections resulting from the decomposition of the tangent bundle induced by the metric $g$
	$$TQ = \D \oplus \D^{\bot},$$
	i.e., the projection to $\D$ and to $\D^{\bot}$ given by $\SP:TQ \rightarrow \D$ and $\mathcal{Q}:TQ \rightarrow \D^{\bot}$, respectively.
	
	The trajectories of the nonholonomic system can be described as geodesics but relative to a non-symmetric connection (see \cite{Lewis98})
	$$\nabla^{nh}_{X} Y = \nabla^{g}_{X} Y + (\nabla_{X}\mathcal{Q} ) Y, \quad X,Y \in \mathfrak{X}(Q),$$
	where $\nabla^{g}$ is the Levi-Civita connection. Therefore, a curve $c:[0,h]\rightarrow Q$ is a trajectory of the kinetic nonholonomic system $(L_{g},\D)$ if and only if it satisfies the equations
	\begin{equation}\label{nh:geodesic:eq}
		\nabla^{nh}_{\dot{c}(t)} \dot{c}(t) = 0, \quad \dot{c}(0)\in\D.
	\end{equation}

 \begin{remark}
     Equation \eqref{nh:geodesic:eq} is equivalent to the Lagrange-d'Alembert equations associated with the Lagrangian function $L_{g}$ (\cite{LMdD1996, GrLeMaMa, Bloch}). These equations are given by
     \begin{equation} \label{LdA}
	\begin{split}
	& \frac{d}{dt}\left(\frac{\partial L_{g}}{\partial \dot{q}^{i}}\right) - \frac{\partial L_{g}}{\partial q^{i}}=\lambda_{a}\mu^{a}_{i}(q)  \\
	& \mu^{a}_{i}(q)\dot{q}^{i}=0,
	\end{split}
	\end{equation}
    where $(q^{i})$ are local coordinates on $Q$, $(q^{i}, \dot{q}^{i})$ are the corresponding local coordinates on $TQ$, $\mu^{a}=\mu_{i}^{a}dq^{i}$ are local $1$-forms on $Q$ such that $\D=\{v_{q}\in T_{q} Q \ | \ \mu^{a}(v_{q})=0)\},$ and $\lambda_{a}$ are Lagrange multipliers that can be determined \david{considering the total derivative} of the constraint equations $\mu^{a}_{i}(q)\dot{q}^{i}=0$.
 \end{remark}

 Throughout the paper, we will denote by $\Gamma_{g}$ the geodesic vector field, i.e., the vector field \david{on $TQ$} whose trajectories are the geodesics of $g$; and by $\Gamma_{(g,\D)}$ the vector field on $\D$ whose trajectories are nonholonomic trajectories, i.e., satisfy equation \eqref{nh:geodesic:eq}.
	
	\subsection{Principal fiber bundles}
	We include here a basic treatment of principal fiber bundles to make the paper more self-contained and fix some notation. The interested reader can read more about the subject in \cite{AM78, KoNo-I, LR89}, for instance.
 
    Let $\Phi:G\times Q\rightarrow Q$ be a \textit{left action} of a Lie group $G$ on a smooth manifold $Q$, denoted by $\Phi(g,q)=g\cdot q=\Phi_g (q)=\Phi_{q}(g)$. The orbit of the action through a point $q\in Q$ is the set $\text{Orb}(q)=\{g\cdot q\; |\;  g\in G\}$. Denote by $\mathfrak{g}$ the Lie algebra of the group $G$. For each element $\xi$ in the Lie algebra there is a vector field on $Q$ called the \textit{infinitesimal generator} of the group action denoted by $\xi_Q$ and defined by $\xi_Q (q)=T_e \Phi_q (\xi)$ \david{where $e$ es the neutral element of $G$}. If we assume that the action $\Phi$ is free and proper we can endow the quotient space $\bar{Q}=Q/G$ with a manifold structure under which the natural projection $\pi:Q\rightarrow \bar{Q}$ is a surjective submersion. Therefore associated to the left action $\Phi$ we have a fibre bundle $\pi$ satisfying the following properties:
	\begin{enumerate}
		\item each $\pi$-fibre, denoted by \david{$Q_{[q]}=\pi^{-1}([q])$}, is an orbit of the action;
		\item the standard fibre is $G$;
		\item the local trivialization $\{U,\psi\}$ of the fibre bundle is \textit{equivariant}, that is, given $U\subseteq \bar{Q}$, an open subset of $\bar{Q}$, the trivialization
		\begin{align*}
			\psi: \pi^{-1}(U) & \rightarrow U\times G \\
			q & \mapsto (\pi(q), \psi_{2} (q) )
		\end{align*}
		satisfies $\psi_{2} (g\cdot q)=g \cdot \psi_{2} (q)$, where we are considering the group multiplication on the right hand-side.
	\end{enumerate}
	The above properties define the \textit{principal $G$-bundle} $(Q, \bar{Q},G,\pi)$, where $Q$ is the \textit{bundle space}, $\bar{Q}$ is the \textit{base space}, $G$ is the \textit{structure group} and $\pi$ is the \textit{projection}. The \textit{vertical space} at points $q\in Q$, denoted by $V_{q}\pi$, form a distribution $V\pi$ on $Q$ and defined as the kernel of $\pi_{*}\equiv T\pi: TQ\rightarrow T\bar{Q}$. The vectors contained in $V\pi$ are called \textit{vertical vectors}. Notice that vertical vectors are just tangent vectors to the orbits of $G$. Explicitly, $$V_{q}\pi=T_{q}(\text{Orb}(q))=\{\xi_{Q}(q) \ | \ \xi\in\mathfrak{g}\}.$$
	A \textit{principal connection} on the principal $G$-bundle is a smooth distribution $\Hcal$ on $Q$ satisfying the following properties:
	\begin{enumerate}
		\item $T_{q}Q=V_{q}\pi\oplus \Hcal_{q}$, for every $q\in Q$;
		\item The distribution is $G$-invariant, i.e., $\Hcal_{g\cdot q}= (\Phi_{g})_{*}(\Hcal_{q})$.
	\end{enumerate}
	$\Hcal$ is called the \textit{horizontal distribution} determined by the connection and the vectors contained in $\Hcal$ are called \textit{horizontal vectors}. Many authors give an alternative equivalent definition of a principal connection as a $\mathfrak{g}$-valued one-form $\omega:TQ\rightarrow \mathfrak{g}$ such that
	\begin{enumerate}
		\item $\omega(\xi_{Q}(q))=\xi$ for every $\xi\in\mathfrak{g}$;
		\item $\omega((T_{q}\Phi_{g})(X))=\text{Ad}_{g}(\omega(X))$ for every $X\in T_{q}Q$.
	\end{enumerate}
	Recall that the adjoint map $\text{Ad}_{g}:\mathfrak{g}\rightarrow\mathfrak{g}$ is the pushforward at the identity of the conjugation map $C_{g}:G\rightarrow G$, $h\mapsto ghg^{-1}$. Such a map $\omega:TQ\rightarrow \mathfrak{g}$ is called the \textit{connection form} determined by the principal connection. Here we may recover the previous properties by defining the \textit{horizontal subspace} at $q$ as $\Hcal_{q}=\text{ker} \ \omega|_{T_{q}Q}$. Conversely, given an horizontal distribution $\Hcal$, we may prove that a $\mathfrak{g}$-valued one-form $\omega$ satisfying $\omega(X)=0$ for every $X\in \Hcal$ and $\omega(\xi_{Q})=\xi$ for every $\xi\in \mathfrak{g}$ is a connection form. The only non-obvious fact needed to prove this is the equality $$(T_{q} \Phi_{g})(\xi_{Q}(q))=[\text{Ad}_{g}(\xi)]_{Q}(g\cdot q).$$
	
	Given a principal connection, every vector $X\in T_{q}Q$ can be uniquely written as $$X=\text{hor}(X)+\text{ver}(X),$$ where $\text{hor}:TQ\rightarrow \Hcal$ and $\text{ver}:TQ\rightarrow V\pi$ are, respectively, the \textit{horizontal} and the \textit{vertical projectors} associated to the decomposition of the tangent space determined by the connection.
	
	The \textit{horizontal lift} of a vector field $X\in\mathfrak{X}(\bar{Q})$ on the base space is the unique horizontal vector field $X^{h}\in\mathfrak{X}(Q)$ on the bundle space that projects onto $X$, i.e., $(T_{q}\pi)(X^{h}(q))=X(\pi(q))$.
	
	A vector field $X\in\mathfrak{X}(Q)$ on the bundle space is \textit{$G$-invariant} if $X(g\cdot q)=(T_{q} \Phi_{g}) (X(q))$. For each $G$-invariant vector field $X$ on $Q$, there exists a unique \textit{reduced vector field} $\overline{X}\in\mathfrak{X}(\bar{Q})$ on the base space such that the following diagram is commutative:
	\[
	\begin{tikzcd}[row sep=2.5em]
		TQ \arrow{rr}{T \pi} && T\bar{Q} \\
		Q \arrow[u,"X"] \arrow{rr}{\pi} && \bar{Q} \arrow[u,"\overline{X}"']
	\end{tikzcd}
	\]
	Moreover, integral curves of $X$ project to integral curves of $\overline{X}$. Conversely, it is well-known in the literature that if we are given the integral curves of the reduced vector field $\overline{X}$ we can \textit{reconstruct} the integral curves of $X$.

\begin{definition}
    The horizontal lift of a curve $\bar{c}:I\to \bar{Q}$ on $\bar{Q}$ is a curve $c:I\to Q$ on $Q$ such that $\pi(c(t))=\bar{c}(t)$ for all $t\in I$ and such that $\dot{c}(t)\in \mathcal{H}_{c(t)}$ for all $t\in I$, where $\mathcal{H}$ is the horizontal distribution associated to a given principal connection.
\end{definition}

In fact, if $q_{0}\in Q$ and $\bar{c}(t_{0})=\pi(q_{0})$ then there exists a unique horizontal lift $c:I\to Q$ of $\bar{c}:I\to \bar{Q}$ such that $c(t_0)=q_{0}$.
 \begin{lemma}\label{horizontal:lift:reparametrization}
     If a curve $c_{1}$ in $\bar{Q}$ is a reparametrization of a curve $c_{2}$ in $\bar{Q}$, then the horizontal lift of $c_{1}$ to $Q$ is a reparametrization of the horizontal lift of the curve $c_{2}$ to $Q$.
 \end{lemma}

 \begin{proof}
     Let $c_{1}:I \to \bar{Q}$, $\varphi:J\to I$ a reparametrization and $c_{2}:J\to \bar{Q}$ defined by $c_{2}(t)=(c_{1} \circ \varphi)(t)$ with $I=[s_{0}, s_{1}]$ and $J=[t_{0}, t_{1}]$.
     Denote by $c_{1}^{h}(t)$ the horizontal lift of the curve $c_{1}$ satisfying $c_{1}^{h}(s_{0}) = q\in Q$.
     The curve $\alpha(t) = (c_{1}^{h} \circ \varphi)(t)$ satisfies $\alpha(t_{0})=(c_{1}^{h} \circ \varphi)(t_0)=c_{1}^{h}(s_{0}) = q$, $\pi(\alpha(t))=(c_{1}\circ \varphi)(t)=c_{2}(t)$ and
     $$\dot{\alpha}(t)=\dot{\varphi}(t)\dot{c}_{1}^{h}(\varphi(t))\in \mathcal{H}_{\alpha(t)},$$
     since $\dot{c}_{1}^{h}( \varphi(t))\in \mathcal{H}_{\alpha(t)}$. Therefore, $\alpha(t)$ is the horizontal lift $c_{2}^{h}$ of $c_{2}:J\to \bar{Q}$ that passes through the point $q$.
 \end{proof}

 \subsection{Riemannian submersions}
	
	Let $(Q,g)$ and $(\bar{Q}, \bar{g})$ be Riemannian manifolds and the map $\pi:Q \rightarrow \bar{Q}$ a surjective submersion. Let $V\pi := \ker T\pi$ and $\D := V^{\bot}\pi$ denote the vertical and horizontal distributions, respectively. Then the map $\pi$ is said to be a Riemannian submersion if
	\begin{equation*}
		\bar{g}(T\pi (X), T\pi (Y)) = g(X, Y),
	\end{equation*}
	for $X, Y\in \D$ (see \cite{ON66,ON67,ON83} for more details).
	Since $\pi$ is a submersion, $T_{q}\pi:(\D_{q}, g(q)|_{\mathcal{D}\times \D})\rightarrow (T_{\pi(q)}\bar{Q}, \bar{g}(\pi(q)))$ is a linear isometry for all $q\in Q$. So, for each $X\in \mathfrak{X}(\bar{Q})$ there is a unique horizontal vector field $X^{h}\in \Gamma(\D)$ such that $T\pi(X^{h})=X$. The vector field $X^{h}$ is the horizontal lift of $X$ by the Riemannian submersion $\pi$.
	
	\begin{lemma}\label{metric:submersion:lemma}
		If $X, Y \in \mathfrak{X}(\bar{Q})$, then
		\begin{enumerate}
			\item $g(X^{h}, Y^{h}) = \bar{g}(X, Y) \circ \pi$.
			\item $\SP([X^{h}, Y^{h}]) = ([X, Y])^{h}$, where $\SP:TQ\rightarrow \D$ is the orthogonal projection.
			\item $\SP(\nabla^{g}_{X^{h}} Y^{h}) = (\nabla^{\bar{g}}_{X} Y)^{h}$, where $\nabla^{g}$ and $\nabla^{\bar{g}}$ are the Levi-Civita connections of $(Q,g)$ and $(\bar{Q},\bar{g})$, respectively.
		\end{enumerate}
	\end{lemma}

	\begin{proof}
		See \cite{ON66,ON67,ON83}.
	\end{proof}

	A simple corollary of the previous lemma is that geodesics $c:I\rightarrow Q$ satisfying $\dot{c}(t)\in \D_{c(t)}$ project under $\pi$ to geodesics in $\bar{Q}$.
	
	Moreover, the tangent lift of geodesics in $Q$ with initial velocity in $\D$ always remain in $\D$. We include a proof of this fact for completeness. We will use the orthogonal projections $\SP:TQ \rightarrow \D$ and $\mathcal{Q}:TQ \rightarrow V\pi$ in the proof of the following theorem.
	
	\begin{proposition}\label{submersion:geodesics}
		Let $\pi:Q \rightarrow \bar{Q}$ be a Riemannian submersion with $\D=V^{\bot}\pi$ and $V\pi=\ker T\pi$ the horizontal and vertical distributions. If $c$ is a geodesic of $Q$ which is horizontal at one point, then it is always horizontal. In fact, the horizontal lift of a geodesic in $(\bar{Q},\bar{g})$ is a geodesic for the Riemannian manifold $(Q,g)$.
	\end{proposition}

	\begin{proof}
		Let $c:[0, h]\rightarrow Q$ be a geodesic of $Q$ satisfying $c(0) = q_{0}, \ \dot{c}(0)\in\D_{q_{0}}$. Consider the projection $\bar{c}:[0,h]\rightarrow \bar{Q}$ defined by $\bar{c} = \pi \circ c$. In the following note that $\dot{c}$ might be written as
		$$\dot{c} = \SP(\dot{c}) + \mathcal{Q}(\dot{c}),$$
		that is the sum of the projection to the horizontal and vertical spaces. The length of the curve $c$ is the functional
		$$\ell (c) = \int_{0}^{h} \|\dot{c} \| \ dt,$$
		and if $h$ is sufficiently small it is the Riemannian distance between $q_{0}:=c(0)$ and $q_{1}:=c(h)$. Attending to the decomposition of $\dot{c}$ we have that
		$$\ell (c)\geqslant \int_{0}^{h} \|\SP(\dot{c}) \| \ dt = \int_{0} ^{h} \| T\pi \circ \SP(\dot{c})\| \ dt,$$
		where the second equality comes from the fact that the metric preserves the inner product of horizontal tangent vectors. Since vertical vectors are by definition in the kernel of $T\pi$, we have that $T\pi(\SP(\dot{c})) = T\pi(\dot{c})$. Therefore, we conclude that
		$$\ell (c) \geqslant \int_{0}^{h} \|(\pi \circ c)'\| \ dt = \ell(\bar{c}).$$
		Denote by $\bar{c}^{h}:[0,h]\rightarrow Q$ the horizontal lift of the curve $\bar{c}$ to the point $q_{0}$. Since $\bar{c}^{h}$ is an horizontal curve, we can easily deduce that $\ell (\bar{c}^{h}) = \ell (\bar{c})$. Therefore $\bar{c}^{h}$ is a curve on $Q$ joining $q_{0}$ and $q_{1}$ and satisfying
		$$\ell (c) \geqslant \ell (\bar{c}^{h}).$$
		Since $c$ is a geodesic we conclude that at least for sufficiently small $h$, $c=\bar{c}^{h}$.
	\end{proof}
	\david{
\begin{remark}
The geodesic relation in Proposition \ref{submersion:geodesics} is also a direct consequence of the following equation:
\[
\nabla^{g}_{X^{h}} Y^{h} = (\nabla^{\bar{g}}_{X} Y)^{h}+\frac{1}{2}
{\mathcal Q}[X^h, Y^h]
\]
(see Lemmas 2 and 3 in \cite{ON66}).   
\end{remark}
    }
	\subsection{Generalized Chaplygin system}\label{chaplygin}
	
	A \textit{generalized Chaplygin system} (\cite{koiller92,Bloch, BKMM96, Cortes}) is a nonholonomic system whose configuration manifold $Q$ is a principal $G$-bundle $\pi:Q\rightarrow Q/G$, the constraint distribution $\D$ determines a principal connection on the bundle with connection form $\omega$ and the Lagrangian function $L:TQ\rightarrow\R$ is a $G$-invariant regular function for the \textit{lifted action} of $G$ on $TQ$, i.e., $$L((T_{q}\Phi_{g}) (v))=L(v), \ \ \ \forall \ v\in T_{q}Q.$$
	
	The \textit{lifted action} is the map $\Phi^{TQ}:G\times TQ\rightarrow TQ$, defined by $\Phi^{TQ}(g,v)=(T_{q} \Phi_{g}) (v)$. If the action is free and proper then the lifted action also is so. 
	
	
	
	On the space $\bar{Q}=Q/G$, one defines a connection (see \cite{koiller92}):
	\begin{equation}\label{reduced:nonholonomic:connection}
		\bar{\nabla}_{\bar{X}} \bar{Y} = T\pi (\nabla^{nh}_{\bar{X}^{h}} \bar{Y}^{h}). \quad \text{for } X, Y \in \mathfrak{X}(\bar{Q}).
	\end{equation}

    Here, $\bar{X}^{h}$, $\bar{Y}^{h}$ are the horizontal lifts of the vector fields $\bar{X}$, $\bar{Y}$, respectively, with respect to the principal connection $\omega$, that is, $\omega(\bar{X}^{h})=\omega(\bar{Y}^{h})=0$ and $\bar{X}^{h}$, $\bar{Y}^{h}$ are $\pi$-projectable over $\bar{X}$, $\bar{Y}$, respectively.
	\begin{proposition}\label{nh:geodesics:horizontal}
		The geodesics of $\nabla^{nh}$ starting in $\D$ project onto geodesics of $\bar{\nabla}$. Conversely, let $\bar{c}$ be a geodesic of $\bar{\nabla}$ and choose $q\in Q$ such that $\pi(q) = \bar{c}(0)$. Then, the geodesic of $\nabla^{nh}$ starting in $q$ is the horizontal lift of $\bar{c}$ with respect to the principal connection $\omega$.
	\end{proposition}
\david{The proof of  Proposition \ref{nh:geodesics:horizontal} is a direct consequence that ${\mathcal D}$ is geodesically invariant with respect to the nonholonomic affine connection $\nabla^{nh}$.}
	
	\section{Main result: The modified metric}
	
	Throughout this section suppose that $(Q,g)$ is a Riemannian manifold, $\D$ is a distribution on $Q$ and consider the kinetic nonholonomic system defined by the Lagrangian function $L_{g}$ and also by the distribution $\D$.
	
	In addition, suppose that the nonholonomic system $(L_{g}, \D)$ is a generalized Chaplygin system with respect to the free and proper action of a Lie group $G$ given by the map $\Phi:G \times Q \rightarrow Q$. This means that the Riemannian metric $g$ is $G$-invariant and $\D$ is the horizontal subbundle associated with a principal connection $\omega$.

    
    Denote by $\bar{Q}$ the reduced space $Q/G$ and by $\pi:Q\rightarrow \bar{Q}$ the projection. Under this hypothesis, and following the discussion in Section \ref{chaplygin}, the nonholonomic vector field reduces to a vector field on $T\bar{Q}$. We will consider on $\bar{Q}$ the Riemannian metric $\bar{g}$ defined by
    $$\bar{g}(\bar{X}, \bar{Y}) \circ \pi = g(\bar{X}^{h}, \bar{Y}^{h}), \text{ for } \bar{X}, \bar{Y} \in \mathfrak{X}(\bar{Q}).$$

Define the gyroscopic tensor $\mathcal{T}$ as a $(1,2)$-tensor field $\mathcal{T}:\mathfrak{X}(\bar{Q}) \times \mathfrak{X}(\bar{Q}) \rightarrow \mathfrak{X}(\bar{Q})$  by
$$\mathcal{T}(\bar{X}, \bar{Y})(\pi(q))=T_{q}\pi\left({\mathcal P}[\bar{X}^h, \bar{Y}^h](q)\right)- [\bar{X}, \bar{Y}](\pi(q)).$$
On the reduced space $\bar{Q}$, take coordinates $(\bar{q}^{a})$. 
The gyroscopic tensor is given in the previous coordinates by
    $$\mathcal{T}\left( \frac{\partial }{\partial \bar{q}^{a}},\frac{\partial }{\partial \bar{q}^{b}} \right) = C_{a b}^{c}(\bar{q}) \frac{\partial }{\partial \bar{q}^{c}}$$
    where $C_{ab}^{c}$ is given by
    $$\SP \left(\left[ \left( \frac{\partial }{\partial \bar{q}^{a}} \right)^{h},\left( \frac{\partial }{\partial \bar{q}^{b}} \right)^{h} \right] \right) = C_{ab}^{c}(\bar{q}) \left( \frac{\partial }{\partial \bar{q}^{c}} \right)^{h},$$ with $\mathcal{P}:TQ\to \D$ the orthogonal projection (see \cite{GaMa} and the references therein). Suppose that $\Gamma_{(g,\D)} \in \mathfrak{X}(\D)$ is the vector field on $\D$ whose trajectories are the nonholonomic trajectories of the Chaplygin system associated with the Lagrangian function $L_{g}$. Denote by $\Omega_{nh}$ the almost symplectic 2-form on $T^{*}\bar{Q}$ given by
    $$\Omega_{nh} = \omega_{\bar{Q}} + \Omega_{\mathcal{T}},$$
    where
    $$\Omega_{\mathcal{T}}(\alpha)(U,V) = \alpha \left( \mathcal{T}\left( T_{\alpha} \pi_{\bar{Q}}(U), T_{\alpha} \pi_{\bar{Q}}(V)
 \right) \right), \text{ for } \alpha\in T^*\bar{Q}, \ U, V \in T_{\alpha}(T^{*} \bar{Q}),$$ $\pi_{\bar{Q}}:T^{*}\bar{Q}\rightarrow \bar{Q}$ is the canonical projection and $\omega_{\bar{Q}}$ denotes the canonical symplectic form on $T^{*}\bar{Q}$. Denote also by $X_{nh}\in \mathfrak{X}(T^{*}\bar{Q})$ the vector field which is obtained as the projection of $\Gamma_{(g,\D)}$ by the smooth map $$\flat_{\bar{g}}\circ T\pi:\D\to T\bar{Q} \to T^{*}\bar{Q},$$
 where $\flat_{\bar{g}}:T\bar{Q} \rightarrow T^{*}\bar{Q}$ is the musical isomorphism associated with the metric $\bar{g}$. $X_{nh}$ is characterized by the following condition
 $$i_{X_{nh}}\Omega_{nh} = dH_{\bar{g}},$$
 with $H_{\bar{g}}:T^{*}\bar{Q}\rightarrow \R$ being the kinetic energy induced by the Riemannian metric $\bar{g}$, that is,
 \begin{equation}\label{kinetic:hamiltonian}
     H_{\bar{g}}(\alpha) = \frac{1}{2}\bar{g}(\sharp_{\bar{g}}(\alpha), \sharp_{\bar{g}}(\alpha)),
 \end{equation}
with $\sharp_{\bar{g}}:T^{*}\bar{Q} \rightarrow T\bar{Q}$ the inverse musical isomorphism $\sharp_{\bar{g}}= (\flat_{\bar{g}})^{-1}$. This almost symplectic 2-form can be used to define in a natural way an associated linear almost Poisson structure (for more details see \cite{GaMa}).

\begin{definition}[\cite{GaMa}]
    A $\phi$-simple system, with $\phi\in C^{\infty}(\bar{Q})$, is a nonholonomic Chaplygin system for which the gyroscopic system satisfies
    \begin{equation*}
        \mathcal{T}(Y,Z)=Z(\phi)Y - Y(\phi)Z, \ \forall \ Y, Z \in \mathfrak{X}(\bar{Q}).
    \end{equation*}
\end{definition}

    $\phi$-simple systems have the remarkable property that the 2-form
    $$\Omega = e^{(\phi\circ \pi_{\bar{Q}})} \Omega_{nh}$$
    is a symplectic structure (see \cite{GaMa}).

    \begin{proposition}\label{vector:bundle:isomorphism}
        If a Chaplygin system is $\phi$-simple then the vector bundle isomorphism $\psi:T^{*}\bar{Q}\rightarrow T^{*}\bar{Q}$ over the identity of $\bar{Q}$ given by
        $$\psi(\alpha_{\bar{q}})=e^{\phi(\pi_{\bar{Q}}(\alpha_{\bar{q}}))} \alpha_{\bar{q}}, \text{ for } \alpha_{\bar{q}}\in T_{\bar{q}}^{*}\bar{Q}$$
        satisfies
        $$\psi^{*}\omega_{\bar{Q}} = \Omega$$
        and
    \begin{equation}\label{hamiltonian:relation}
        H_{g_{can}} = H_{\bar{g}} \circ \psi^{-1},
    \end{equation}
    where $g_{can}$ is the Riemannian metric defined by
    \begin{equation}\label{canonical:reduced:metric}
        g_{can}=e^{2\phi} \bar{g},
    \end{equation}
    and $H_{g_{can}}$ is the kinetic energy on $T^{*}\bar{Q}$ associated with $g_{can}$.
    \end{proposition}

    \begin{proof}
        From Theorem 3.21 and Lemma 3.24 in \cite{GaMa}, a system is $\phi$-simple if and only if 
        $$\Omega_{\mathcal{T}}=\theta_{\bar{Q}}\wedge d(\phi\circ\pi_{\bar{Q}}),$$
        where $\theta_{\bar{Q}}$ is the Liouville $1$-form on $T^{*}Q$. This implies that
        $$\Omega = e^{(\phi\circ\pi_{\bar{Q}})} \omega_{\bar{Q}} + e^{(\phi\circ\pi_{\bar{Q}})} \theta_{\bar{Q}}\wedge d(\phi\circ\pi_{\bar{Q}}).$$

        On the other hand, it is easy to show that
        $$\psi^{*} \theta_{\bar{Q}} =e^{(\phi\circ\pi_{\bar{Q}})} \theta_{\bar{Q}}.$$
        From here, it follows that $$\psi^{*}\omega_{\bar{Q}} = -d(\psi^{*} \theta_{\bar{Q}}) = e^{(\phi\circ\pi_{\bar{Q}})} \omega_{\bar{Q}} + e^{(\phi\circ\pi_{\bar{Q}})} \theta_{\bar{Q}}\wedge d(\phi\circ\pi_{\bar{Q}})=\Omega.$$

        In addition, a direct computation using the definitions of $H_{g_{can}}$ and $H_{\bar{g}}$ proves the equality in \eqref{hamiltonian:relation}.
    \end{proof}

    Now, we can prove the following result.
    \begin{corollary}\label{reparametrization:corollary}
        If $X_{g_{can}}\in \mathfrak{X}(T^{*}\bar{Q})$ is the geodesic vector field of $g_{can}$ in $T^{*}\bar{Q}$ then the vector fields $e^{-(\phi\circ\pi_{\bar{Q}})}X_{nh}$ and $X_{g_{can}}$ are $\psi$-related. So, the trajectories of $X_{nh}$ are reparametrizations of geodesics with respect to the metric $g_{can}$.
    \end{corollary}

    \begin{proof}
        From Theorems 3.11 and 3.21 in \cite{GaMa}, we have that the dynamical system $X_{nh}$ is Hamiltonizable up to reparametrization. In fact, it satisfies
        $$i_{e^{-(\phi\circ\pi_{\bar{\footnotesize Q}})}X_{nh}}\Omega = dH_{\bar{g}}.$$
        Using the previous proposition, we conclude that
        $$i_{e^{-(\phi\circ\pi_{\bar{\footnotesize Q}})}X_{nh}}\psi^{*}\omega_{\bar{Q}} = dH_{\bar{g}},$$
        which leads to
        $$i_{\left( e^{-(\phi\circ\pi_{\bar{Q}})} T\psi \circ X_{nh} \circ \psi^{-1}\right)} \omega_{\bar{Q}} = d(H_{\bar{g}}\circ \psi^{-1}).$$

        Thus, using again Proposition \ref{vector:bundle:isomorphism}, it follows that the vector field $e^{-(\phi\circ\pi_{\bar{Q}})} T\psi \circ X_{nh} \circ \psi^{-1} \in \mathfrak{X}(T^{*}\bar{Q})$ is just $X_{g_{can}}$.

        This implies that
        $$X_{g_{can}}\circ \psi = T\psi \circ \left( e^{-(\phi\circ\pi_{\bar{Q}})} X_{nh} \right) $$
        and the vector fields $e^{-(\phi\circ\pi_{\bar{Q}})} X_{nh}$ and $X_{g_{can}}$ are $\psi$-related. So, if $\bar{\gamma}:I\to T^*\bar{Q}$ is an integral curve of $e^{-(\phi\circ\pi_{\bar{Q}})} X_{nh}$ then $\psi\circ \bar{\gamma}: I\to T^*\bar{Q}$ is an integral curve
        of $X_{g_{can}}$. Thus, using that the trajectories of $e^{-(\phi\circ\pi_{\bar{Q}})} X_{nh}$ and $X_{g_{can}}$ are the
        $\pi_{\bar{Q}}$-projections of their integral curves and the fact that $\pi_{\bar{Q}} \circ \psi = \pi_{\bar{Q}}$, we conclude that the trajectories of $e^{-(\phi\circ\pi_{\bar{Q}})} X_{nh}$ are just the geodesics of the metric $g_{can}$. Finally , since the trajectories of $e^{-(\phi\circ\pi_{\bar{Q}})} X_{nh}$ are reparametrizations of the trajectories of $X_{nh}$, we deduce the result.
    \end{proof}

    Next, we will prove the main result of the paper.
    \begin{theorem}\label{extension:main:thm}
        Let $(Q,g, G, \D)$ be a kinetic Chaplygin system and $\pi:Q\rightarrow \bar{Q}$ the projection onto the quotient space. If the system is $\phi$-simple, there exists a Riemannian metric $h$ on $Q$ such that all its geodesics with initial velocity in $\D$ are reparametrizations of the nonholonomic trajectories associated with $(L_{g}, \D)$. In fact, if $V\pi=\ker T\pi$ and $g_{can}$ is the Riemannian metric on $\bar{Q}$ defined by $g_{can}=e^{2\phi}\bar{g}$ then $h$ may be chosen as follows
        \begin{enumerate}
			\item $h(\bar{X}^{h}, \bar{Y}^{h}) = g_{can}(\bar{X}, \bar{Y}), \forall \bar{X}, \bar{Y}\in \mathfrak{X}(\bar{Q})$;
			\item $h(Z, W) = g(Z, W), \forall Z,W\in \Gamma(V\pi)$;
			\item $h(\bar{X}^{h}, Z)= 0, \forall \bar{X}\in \mathfrak{X}(\bar{Q}), \ Z\in \Gamma(V\pi)$.
		\end{enumerate}
        Here, $\bar{X}^{h}, \bar{Y}^{h}$ are the horizontal lifts of $\bar{X}, \bar{Y}$, respectively, by the principal connection associated with the Chaplygin system.
    \end{theorem}

    \begin{proof}
        The proof has two parts. First, we will prove that the projection $\pi:(Q,h)\to (\bar{Q}, g_{can})$ is a Riemannian submersion. 
		
		The metric $h$ is indeed a Riemannian metric, it is clearly a symmetric $(0,2)$-type smooth tensor. It is positive definite since if $q\in Q$ and $Z \in T_{q} Q$, $Z_{q}\neq 0$, then
		$$h(Z_{q},Z_{q}) = h(P_{\D}(Z_{q}) + P_{V\pi}(Z_{q}), P_{\D}(Z_{q}) + P_{V\pi}(Z_{q}))$$
		where $P_{\D}:TQ \rightarrow \D$ and $P_{V\pi}:TQ \rightarrow V\pi$ are the projections associated with the decomposition $TQ = \D \oplus V\pi$. Hence,
		$$h(Z_{q},Z_{q}) = g_{can}(T\pi (P_{\D}(Z_{q})), T\pi ( P_{\D}(Z_{q}))) + g(P_{V\pi}(Z_{q}), P_{V\pi}(Z_{q}))>0,$$
		by positive definiteness of $g_{can}$ and $g$ (note that if $Z_{q}\neq 0$ then $T\pi (P_{\D}(Z_{q}))\neq 0$ or $P_{V\pi}(Z_{q})\neq 0$). In addition, with respect to the Riemannian metric $h$, $\D = V^{\bot}\pi$. Hence, $\pi:(Q,h) \rightarrow (\bar{Q}, g_{can})$ is a Riemannian submersion since
		$$h(X,Y) = g_{can} (T\pi(X), T\pi(Y)),$$ for $X,Y\in \D$.
  
        Therefore, by Proposition \ref{submersion:geodesics}, the tangent lift of all geodesics of $(Q,h)$ with initial velocity in $\D$ stay in $\D$ and such geodesics are the horizontal lift of geodesics of $(\bar{Q}, g_{can})$.

        Moreover, by Corollary \ref{reparametrization:corollary}, reparametrizations of geodesics with respect to the metric $g_{can}$ are trajectories of $X_{nh}$ and, using Proposition \ref{nh:geodesics:horizontal}, the horizontal lift of the latter with respect to the nonholonomic principal connection $\omega$ are the nonholonomic trajectories. Note that the horizontal lift with respect to $\omega$ coincides with the horizontal lift with respect to the Riemannian submersion $\pi:(Q,h)\rightarrow (\bar{Q}, g_{can})$. This fact, Proposition \ref{submersion:geodesics} and Lemma \ref{horizontal:lift:reparametrization} imply that the nonholonomic trajectories must be reparametrizations of geodesics of $h$ with initial velocity in $\D$.
    \end{proof}

    \begin{remark}
		The construction of a Riemannian metric $h$ in the conditions of the Theorem above is by no means unique. Any metric $h'$ rendering $\pi:(Q, h')\to (\bar{Q}, g_{can})$ a Riemannian submersion will have the same geodesics with initial velocity in $\D$. In particular, the restriction of $h'$ to vertical vector fields may be given by any bundle metric on the vertical bundle.
	\end{remark}

    \begin{definition}
        The Riemannian metric $h$ in Theorem \ref{extension:main:thm} is called the \textit{principal Riemannian metric} with respect to $(Q,g)$ and $(\bar{Q},g_{can})$.
    \end{definition}

    Denote the norm with respect to the Riemannian metric $h$ by $\|\cdot \|_{h}$ and by $d:Q\times Q \to 
    \R$ the Riemannian distance induced by $h$, that is,
    $$d(p,q) = \inf \{ \ell(c) \ | \ c\in \Omega(q,p) \},$$ where $\Omega(q,p)$ is the set of all piecewise smooth curves segments in $Q$ from $q$ to $p$ and, for each $c:[0,T]\to Q$ in $\Omega(q,p)$ the length $\ell(c)$ of the segment $c$  is given by 
    $$\ell(c) = \int_{0}^{T} \| \dot{c} \|_{h} \ dt.$$
    Then, the following corollary of Theorem \ref{extension:main:thm} asserts that the length of a nonholonomic trajectory between two nearby points $q$ and $p$ is just the distance from $q$ to $p$. In other words, nonholonomic trajectories between two nearby points minimize the Riemannian length. The proof of the corollary is an immediate consequence of Theorem \ref{extension:main:thm} together with the facts that geodesics locally minimize the Riemannian distance and that the length of a curve is invariant under time reparametrization.

    \begin{corollary}\label{Riemannian:distance:cor}
        Let $(Q,g, G, \D)$ be a kinetic Chaplygin system. Suppose that the system is $\phi$-simple and $h$ is the principal Riemannian metric.  If $c:[0,T]\to Q$ is a nonholonomic trajectory there exists $0<t_{0}<T$ such that for all $t<t_{0}$ the length of the curve $c:[0,t]\to Q$ is just the Riemannian distance associated with $h$ between $c(0)$ and $c(t)$.
    \end{corollary}

\section{$\phi$-simple Chaplygin mechanical systems}



    In our final result, we will see that the inclusion of an invariant potential function in the picture does not substantially change the previous results. Consider a smooth  $G$-invariant potential function $V$ on $Q$. By symmetry, this function is characterized by the existence of a smooth potential function $\bar{V}$ on $\bar{Q}$ such that $V = \bar{V} \circ \pi.$

    The trajectories of the Lagrangian system associated with a Riemannian metric $g$ and a potential function $V$, i.e., the trajectories of the Lagrangian system whose Lagrangian function is given by
    $$L_{(g,V)}=L_{g} - V\circ \tau_{Q}$$
    are commonly designated by trajectories of the mechanical system associated with $g$ and $V$.
    
    We are interested in extending Theorem \ref{extension:main:thm} to the case where there exists a potential function and thus the trajectories of the nonholonomic system $(L_{(g,V)}, \D)$ will no longer be reparametrizations of geodesics but of trajectories of a mechanical system with respect to a different Riemannian metric. The proof makes use of the previous results and adapts them to the presence of the potential.

    \subsection{Riemannian submersions and mechanical Lagrangian systems with basic potential}

    First of all, one is interested in obtaining an analogous result to Proposition \ref{submersion:geodesics} for a mechanical system with Lagrangian function $L_{(g,V)}$ and for a Riemannian submersion $\pi:(Q,g)\to (\bar{Q},\bar{g})$. On one hand, it is well-known (see, e.g. \cite{AM78}) that a curve $c:I \rightarrow Q$ is a trajectory of the mechanical system determined by $L_{(g,V)}$ if and only if it satisfies
    $$\nabla^{g}_{\dot{c}} \dot{c} = - \text{grad}_{g} V(c(t)),$$
    where $\text{grad}_g V$ is the gradient vector field of $V$ with respect to the metric $g$, that is,
    $$g(\text{grad}_{g} V (q), X)=dV(q) (X), \text{ for } X\in T_{q}Q.$$

    So, using Lemma \ref{metric:submersion:lemma} and the $G$-invariance of $V$, it is not difficult to prove that
    $$\nabla^{\bar{g}}_{\dot{\bar{c}}} \dot{\bar{c}} = - \text{grad}_{\bar{g}}\bar{V}(\bar{c}(t)),$$
    where $\bar{c}=\pi \circ c$ and $c$ is a horizontal trajectory, that is, $\dot{c}(t)\in V_{c(t)}^{\bot}\pi$, $\forall t$ and $\bar{V}\in C^{\infty}(\bar{Q})$ is defined by $V=\bar{V}\circ \pi$. Therefore, horizontal trajectories of the mechanical system associated with $L_{(g,V)}$ project to trajectories of the mechanical system associated with $L_{(\bar{g}, \bar{V})}$.

    On the other hand, the vector field $\Gamma_{(g,V)}$ on $TQ$ generating the dynamics associated with $L_{(g,V)}$, is tangent to the horizontal distribution $V^{\bot}\pi$ implying that if $c$ is a mechanical trajectory which is horizontal at one point then it is horizontal at all points. This can be seen by noting that this vector field can be written as
    $$\Gamma_{(g,V)} = \Gamma_{g} - (\text{grad}_{g} V)^{\textbf{v}}$$
    where $\Gamma_{g}$ is the geodesic vector field of $g$ and
    $$(\text{grad}_{g} V)^{\textbf{v}} (v_{q})=\left.\frac{d}{dt}\right|_{t=0}(v_{q} + t \text{grad }V (q)), \text{ for } v_{q}\in T_{q} Q.$$


    

    Now, it is easy to see that $\text{grad}_{g} V$ is a horizontal vector field. In fact, if $X \in V_q \pi = \ker (T_{q}\pi)$ then, using that $V= \bar{V} \circ\pi$, we deduce that 
    $$g(\text{grad }V (q), X) = X(V) = X(\bar{V} \circ\pi)= 0 .$$
    Thus, $(\text{grad}_{g} V)^{\textbf{v}}|_{\D} \in \mathfrak{X}(\D)$. On the other hand, by Proposition \ref{submersion:geodesics}, we have that $\Gamma_{g}|_{\D}\in\mathfrak{X}(\D)$. Therefore, $\Gamma_{(g,V)}|_{\D} \in \mathfrak{X}(\D)$.

    In summary, we have proved the following result.

    \begin{proposition}\label{mechanical:submersion:prop}
        Let $L_{(g,V)}:TQ\to \R$ be a mechanical Lagrangian function and $\pi:(Q,g)\to (\bar{Q}, \bar{g})$ a Riemannian submersion such that $V\in C^{\infty}(Q)$ is a $\pi$-basic function, that is, there exists $\bar{V}\in C^{\infty}(\bar{Q})$ and $V=\bar{V}\circ \pi$. Denote by $V\pi=\ker T\pi$ and $\D=V^{\bot}\pi$ the vertical and horizontal distributions, respectively, associated with the Riemannian submersion $\pi$.
        \begin{enumerate}
            \item If $c: I\to Q$ is a horizontal trajectory for the mechanical system determined by $g$ and $V$ then $\bar{c}= \pi \circ c$ is a trajectory for the mechanical system determined by $\bar{g}$ and $\bar{V}$.
            \item If $c: I\to Q$ is a trajectory for $L_{(g,V)}$ which is horizontal at one point, then it is always horizontal. In fact, the horizontal lift of a trajectory for $L_{(\bar{g},\bar{V})}$ is a trajectory for $L_{(g,V)}$.
        \end{enumerate}
    \end{proposition}
    

    \subsection{The main result of the section}

    Before presenting the main result of this section, there are a few considerations about $\phi$-simple mechanical systems that we must recall.

   The notion of $\phi$-simplicity is independent of the potential function. So, if $(Q,g,G,\D)$ is a $\phi$-simple kinetic nonholonomic Chaplygin system, then the mechanical system continues to be $\phi$-simple under the addition of a $G$-invariant potential (see \cite{GaMa}).

    Following \cite{GaMa}, the framework described at the beginning of Section 3, still holds under the addition of a $G$-invariant potential function.

    Let $(Q,g,V, G, \D)$ be a $\phi$-simple Chaplygin mechanical system. Consider the Hamiltonian function $H_{(g, V)}:T^{*}Q\to \R$ given by
    $$H_{(g, V)}(\alpha_{q}) = \frac{1}{2}g(\sharp_{g}(\alpha_{q}), \sharp_{g}(\alpha_{q})) + V(q), \text{ for } \alpha_{q} \in T_{q}^{*}Q.$$
    Next, we define the reduced Hamiltonian function $H_{(\bar{g}, \bar{V})}:T^{*}\bar{Q}\to \R$ as follows
    $$H_{(\bar{g}, \bar{V})}(\alpha_{\bar{q}}) = \frac{1}{2}\bar{g}(\sharp_{\bar{g}}(\alpha_{\bar{q}}), \sharp_{\bar{g}}(\alpha_{\bar{q}})) + \bar{V}(\bar{q}), \text{ for } \alpha_{\bar{q}} \in T_{\bar{q}}^{*}\bar{Q},$$
    where $\bar{g}$ is the reduced Riemannian metric and $\bar{V}\in C^{\infty}(\bar{Q})$ is such that $V=\bar{V}\circ \pi$.
        
   Now, we can consider the symplectic vector bundle isomorphism $\psi:(T^{*}\bar{Q}, \omega_{\bar{Q}})\to (T^{*}\bar{Q}, \Omega)$ in Proposition \ref{vector:bundle:isomorphism} and, consequently, we can define the Riemannian metric $g_{can}$ on $\bar{Q}$ as in this proposition. Therefore we can introduce a new Hamiltonian function $H_{(g_{can}, \bar{V})}:T^{*}\bar{Q}\to \R$ given by
    $$H_{(g_{can}, \bar{V})}(\alpha_{\bar{q}}) = \frac{1}{2}g_{can}(\sharp_{g_{can}}(\alpha_{\bar{q}}), \sharp_{g_{can}}(\alpha_{\bar{q}})) + \bar{V}(\bar{q}),$$
    which is the mechanical energy associated with the Riemannian metric determined by $g_{can}$ and the potential energy $\bar{V}\in C^{\infty}(\bar{Q})$.
    
    Analogously to the case without potential, this Hamiltonian function satisfies $$H_{(g_{can}, \bar{V})}=H_{(\bar{g}, \bar{V})} \circ \psi^{-1}.$$

    Since the system is $\phi$-simple, the reduced nonholonomic vector field $X_{nh}$ is still Hamiltonizable with the addition of the $G$-invariant potential function (see Theorems 3.11 and 3.21 in \cite{GaMa}). Therefore, the proof of Corollary \ref{reparametrization:corollary} holds by replacing the Hamiltonian function $H_{\bar{g}}$ with $H_{(\bar{g}, \bar{V})}$.


    So, we deduce the following result.
    \begin{corollary}\label{reparametrization:mechanical:system}
        If $X_{(g_{can}, \bar{V})}\in \mathfrak{X}(T^{*}\bar{Q})$ is the dynamical vector field associated with the mechanical Hamiltonian function $H_{(g_{can}, \bar{V})}$, then the vector field $e^{-(\phi\circ\pi_{\bar{Q}})}X_{nh}$ and $X_{(g_{can}, \bar{V})}$ are $\psi$-related. So, the trajectories of the reduced nonholonomic vector field $X_{nh}$ are reparametrizations of trajectories of the Hamiltonian system with mechanical Hamiltonian function $H_{(g_{can}, \bar{V})}$.
    \end{corollary}

    Now, using Proposition \ref{mechanical:submersion:prop}, Corollary \ref{reparametrization:mechanical:system} and proceeding as in the proof of Theorem \ref{extension:main:thm}, we can prove a version of this last theorem for the more general case of a $\phi$-simple Chaplygin mechanical system.

    \begin{theorem}\label{potential:thm}
        Let $(Q,g, V, G, \D)$ be a Chaplygin system with $G$-invariant potential $V$. If the system is $\phi$-simple, there exists a Riemannian metric $h$ on $Q$ such that all the mechanical trajectories associated with $h$ and $V$ with initial velocity in $\D$ are reparametrizations of the nonholonomic trajectories associated with $(L_{(g,V)}, \D)$. In fact, $h$ may be chosen as follows
        \begin{enumerate}
			\item $h(\bar{X}^{h}, \bar{Y}^{h}) = g_{can}(\bar{X}, \bar{Y}), \forall \bar{X}, \bar{Y}\in \mathfrak{X}(\bar{Q})$;
			\item $h(Z, W) = g(Z, W), \forall Z,W\in \Gamma(V\pi)$;
			\item $h(\bar{X}^{h}, Z)= 0, \forall \bar{X}\in \mathfrak{X}(\bar{Q}), \ Z\in \Gamma(V\pi)$.
		\end{enumerate}
        Here, $\bar{X}^{h}, \bar{Y}^{h}$ are the horizontal lifts of $\bar{X}, \bar{Y}$, respectively, by the principal connection associated with the Chaplygin system.
    \end{theorem}


        

    \begin{remark}\label{weaker:condition:remark}
        Note that a more general version of Theorem \ref{potential:thm} may be proved if we replace, in the statement of the theorem, the hypothesis that the Chaplygin system is $\phi$-simple with the following weaker condition: there exists a smooth function $\phi\in C^{\infty}(\bar{Q})$ such that the trajectories of the reduced nonholonomic vector field $X_{nh}$ are reparametrizations of the trajectories of the mechanical system with Lagrangian function $L_{(g_{can}, \bar{V})}$, where $g_{can}= e^{2\phi} \bar{g}$. In particular, this condition holds if $X_{nh}$ is just the dynamics in the Hamiltonian side associated with the mechanical system whose Lagrangian function is $L_{(\bar{g}, \bar{V})}$.
    \end{remark}

    \section{Examples}

    \begin{example}\label{example1}[Vertical rolling disk]
		The vertical rolling disk or the rolling penny is one of the most relevant examples of a nonholonomic kinetic system. In the configuration manifold $Q=\R^{2}\times \Es^{1} \times \Es^{1}$, we introduce the local coordinate system $(x,y,\theta,\varphi)$ and the kinetic Lagrangian function $L_{g}:TQ\rightarrow \R$ given by
		\begin{equation*}
			L_{g}(x,y,\varphi,\theta,\dot{x},\dot{y},\dot{\varphi},\dot{\theta})=\frac{1}{2}m(\dot{x}^2+\dot{y}^2)+\frac{1}{2}I\dot{\theta}^2+\frac{1}{2}J\dot{\varphi}^2,
		\end{equation*}
		where
		$$g = m dx^{2} + m dy^{2} + I d\theta^{2} + J d\varphi^{2}$$
		and subjected to the distribution $\D\subseteq TQ$ determined by the local expression
		\begin{equation*}
			\dot{x}=R \dot{\theta} \cos \varphi, \quad \dot{y}= R \dot{\theta} \sin \varphi,
		\end{equation*}
		where $R$ is the radius of the disk, $m$ is its mass of the disk and $I$, $J$ are the inertia.
		
		It is well-known that under the translation
		\begin{equation*}
			\begin{split}
				\Phi: \R^{2} \times Q &  \longrightarrow Q \\
				((a,b), (x,y,\theta,\varphi)) & \mapsto (x+a, y+b, \theta, \varphi)
			\end{split}
		\end{equation*}
		the kinetic nonholonomic system $(L_{g}, \D)$ is a Chaplygin system with principal bundle projection $\pi: Q \rightarrow \bar{Q}:=\Es^{1} \times \Es^{1}$ and the gyroscopic tensor $\mathcal{T}$ vanishes identically (see \cite{GaMa}). So, the reduced equations are geodesics with respect to the reduced metric $\bar{g}$:
		\begin{equation*}
			\bar{g} = (I+mR^{2}) d\theta^{2} + J d\varphi^{2}
		\end{equation*}		
		on $\bar{Q}$. Therefore, we are in the conditions described in Remark \ref{weaker:condition:remark} and we may define the principal Riemannian metric $h$ on $Q$, i.e., the one such that $h|_{\D \times \D}=g|_{\D \times \D},$, $h|_{V\pi \times V\pi}=g|_{V\pi \times V\pi}$ and $h|_{\D\times V\pi}=0$. Let $\{e_{1}, e_{2}, e_{3}, e_{4}\}$ be the basis of vector fields on $Q$ defined by
		$$e_{1} = \frac{\partial}{\partial \theta} + R \cos \varphi\frac{\partial}{\partial x} + R \sin \varphi \frac{\partial}{\partial y}, \ e_{2}= \frac{\partial}{\partial \varphi}, \ e_{3} = \frac{\partial}{\partial x}, \ e_{4} = \frac{\partial}{\partial y}$$
		so that $\D = \langle \{e_{1}, e_{2}\} \rangle$ and $V\pi = \langle \{e_{3}, e_{4}\} \rangle$.
		Therefore, in $V\pi\times V\pi$ we have
        \begin{equation}\label{aux:h:V}
            h(\frac{\partial}{\partial x}, \frac{\partial}{\partial x}) = m, \ h(\frac{\partial}{\partial y},\frac{\partial}{\partial y}) = m, \  h(\frac{\partial}{\partial x}, \frac{\partial}{\partial y}) = 0.
        \end{equation}
		In $\D\times \D$ we must have
		$$h(e_{1},e_{1})=g(e_{1},e_{1})=I+mR^{2}, \ h(e_{2}, e_{2}) = g(e_{2},e_{2})=J, \ h(e_{1}, e_{2}) = g(e_{1}, e_{2}) = 0.$$
		So, using \eqref{aux:h:V} we may write the following three equations
		\begin{equation*}
			\begin{split}
				& h(\frac{\partial}{\partial \theta}, \frac{\partial}{\partial \theta}) + 2 R \cos \varphi \ h(\frac{\partial}{\partial x}, \frac{\partial}{\partial \theta}) + 2 R \sin \varphi \ h(\frac{\partial}{\partial y}, \frac{\partial}{\partial \theta}) = I \\
				& h(\frac{\partial}{\partial \varphi}, \frac{\partial}{\partial \varphi}) = J \\
				& h(\frac{\partial}{\partial \theta}, \frac{\partial}{\partial \varphi}) + R \cos \varphi \  h(\frac{\partial}{\partial x}, \frac{\partial}{\partial \varphi}) + R \sin \varphi \ h(\frac{\partial}{\partial y}, \frac{\partial}{\partial \varphi}) = 0.
			\end{split}
		\end{equation*}
		Finally, from the condition that $\D$ and $V\pi$ must be orthogonal under $h$, we deduce
		\begin{equation*}
			\begin{split}
				& h(\frac{\partial}{\partial \theta}, \frac{\partial}{\partial x}) = -mR\cos \varphi, \quad h(\frac{\partial}{\partial \theta}, \frac{\partial}{\partial y}) = - mR\sin \varphi \\
				 & h(\frac{\partial}{\partial \varphi}, \frac{\partial}{\partial x}) = 0, \quad h(\frac{\partial}{\partial \varphi}, \frac{\partial}{\partial y}) = 0.
			\end{split}
		\end{equation*}
		Inserting the last four equations into the previous three, we deduce
		$$h(\frac{\partial}{\partial \theta}, \frac{\partial}{\partial \theta}) = I + 2mR^{2}, \quad h(\frac{\partial}{\partial \theta}, \frac{\partial}{\partial \varphi}) = 0.$$
		Therefore, the modified metric $h$ is
		\begin{equation*}
				h = m dx^{2} + m dy^{2} + (I + 2 mR^{2}) d\theta^{2} + J d\varphi^{2} - mR \cos \varphi \ dx d\theta - mR \sin\varphi \ dy d\theta.
		\end{equation*}
		
		The geodesic equations for the Riemannian metric $h$ are
        \begin{equation*}
            \begin{split}
            \ddot{x} & = -\frac{mR^{2}}{2(I + mR^{2})}\sin( 2\varphi) \dot{x}\dot{\varphi} + \frac{mR^{2}}{I + mR^{2}}\cos^{2}(\varphi) \dot{y}\dot{\varphi} - R \sin(\varphi)\dot{\theta}\dot{\varphi} \\
            \ddot{y} & = -\frac{mR^{2}}{I + mR^{2}}\sin^{2}(\varphi) \dot{x}\dot{\varphi} + \frac{mR^{2}}{2(I + mR^{2})}\sin(2\varphi) \dot{y}\dot{\varphi} + R \cos(\varphi)\dot{\theta}\dot{\varphi} \\
            \ddot{\theta} & = -\frac{mR}{I + mR^{2}}\sin(\varphi)\dot{x}\dot{\varphi} + \frac{m R}{I + mR^{2}}\cos(\varphi) \dot{y}\dot{\varphi} \\
            \ddot{\varphi} & = \frac{mR}{J}\sin(\varphi)\dot{x}\dot{\theta}-\frac{mR}{J}\cos(\varphi)\dot{y}\dot{\theta}
        \end{split}
        \end{equation*}
        and their solutions with initial velocity in $\D$ satisfy
        \begin{equation*}
            \begin{split}
            \ddot{x} & = - R\sin(\varphi)\dot{\theta}\dot{\varphi} \\
            \ddot{y} & = R\cos(\varphi)\dot{\theta}\dot{\varphi} \\
            \ddot{\theta} & = 0 \\
            \ddot{\varphi} & = 0.
        \end{split}
        \end{equation*}
        On the other hand, the equations of the nonholonomic trajectories can be found using equations \eqref{LdA}, from where we obtain
        \begin{equation*}
	\begin{split}
	m\ddot{x} & = \lambda_{1}  \\
	m\ddot{y} & = \lambda_{2} \\
	I\ddot{\theta} & =-\lambda_{1} R \cos \varphi-\lambda_{2} R \sin \varphi \\
	J \ddot{\varphi} & =0 \\
	\end{split}
	\end{equation*}
        which after determination of the Lagrange multipliers give precisely the previous set of equations.
	\end{example}

    \begin{example}\label{example2}(The nonholonomic particle)
        Consider a nonholonomic system with configuration space $Q=\R^{3}$, Lagrangian function given by $$L_{g} = \frac{1}{2}(\dot{x}^{2}+ \dot{y}^{2} + \dot{z}^{2})$$
        and subjected to the constraint $\dot{z} = y \dot{x}$.
        This system, known as the nonholonomic particle, is a Chaplygin system together with the group action of $\mathbb{R}$ acting by translations on $z$. The reduced configuration space is $\bar{Q}=\R^{2}$ with principal bundle projection $\pi:Q\rightarrow \bar{Q}$ given by $\pi(x,y,z)=(x,y)$.

        Moreover, from \cite{GaMa}, it is also a $\phi$-simple system with
        $$\phi(x,y)=-\frac{1}{2}\ln(1+y^{2}).$$
        This system has the corresponding reduced Riemannian metric $\bar{g}$ 
        on $\bar{Q}$ given by
        $$\bar{g}=(1 + y^{2}) dx^{2} + dy^{2}.$$
        Therefore, the conformal change $g_{can}=e^{2\phi}\bar{g}$ gives the Riemannian metric
        $$g_{can}=dx^{2} + \frac{dy^{2}}{1 + y^{2}}.$$

        The principal Riemannian metric on $Q$ figuring in Theorem \ref{extension:main:thm} is
        $$h = (1+y^{2}) dx^{2} + \frac{dy^{2}}{1+y^{2}} + dz^{2} - y dx dz.$$

        The geodesic equations with respect to $h$ are
        \begin{equation*}
            \begin{split}
                \ddot{x} & = -y \dot{x}\dot{y} + \dot{y}\dot{z} \\
                \ddot{y} & = \frac{(y^{4} + 2 y^{2} + 1 )y\dot{x}^2 - (y^{4} + 2y^{2} + 1)\dot{x}\dot{z} + y \dot{y}^2}{1 + y^2} \\
                \ddot{z} & = -y^2 \dot{x}\dot{y} + \dot{x}\dot{y} + y \dot{y}\dot{z}
            \end{split}
        \end{equation*}
        Their solutions with initial velocity in $\D$ satisfy
        \begin{equation}\label{h:equations}
            \begin{split}
                \ddot{x} & = 0 \\
                \ddot{y} & = \frac{y \dot{y}^2}{1 + y^2} \\
                \ddot{z} & =  \dot{x}\dot{y}
            \end{split}
        \end{equation}
        On the other hand, using \eqref{LdA}, we deduce that the nonholonomic trajectories are solutions of the equations
        \begin{equation}\label{nh:equations}
            \begin{split}
                \ddot{x} & = -y \frac{\dot{x}\dot{y}}{1+y^{2}} \\
                \ddot{y} & = 0 \\
                \ddot{z} & = \frac{\dot{x}\dot{y}}{1+y^{2}} 
            \end{split}
        \end{equation}
        with Lagrange multiplier already determined.
        
        Indeed, if we rewrite \eqref{nh:equations} as a first order system of differential equations, we obtain the equations
        \begin{equation*}
            \begin{split}
                \dot{x} & = v_{x} \quad \quad \dot{v_x} = -y \frac{v_{x} v_{y}}{1+y^{2}} \\
                \dot{y} & = v_{y} \quad \quad \dot{v_y} = 0 \\
                \dot{z} & = v_{z} \quad \quad \dot{v_z} = \frac{v_{x} v_{y}}{1+y^{2}}.
            \end{split}
        \end{equation*}
        whose solutions are trajectories of the vector field $\Gamma_{(g,\D)}$. Consequently, the trajectories of the vector field $\sqrt{1+y^{2}}\Gamma_{(g,\D)}$ satisfy
        \begin{equation*}
            \begin{split}
                \dot{x} & = \sqrt{1+y^{2}}v_{x} \quad \quad \dot{v_x} = -y \frac{v_{x} v_{y}}{\sqrt{1+y^{2}}} \\
                \dot{y} & = \sqrt{1+y^{2}}v_{y} \quad \quad \dot{v_y} = 0 \\
                \dot{z} & = \sqrt{1+y^{2}}v_{z} \quad \quad \dot{v_z} = \frac{v_{x} v_{y}}{\sqrt{1+y^{2}}}.
            \end{split}
        \end{equation*}
        Considering the change of variables $\bar{v}_{x}=\sqrt{1+y^{2}}v_{x}$, $\bar{v}_{y}=\sqrt{1+y^{2}}v_{y}$ and $\bar{v}_{z}=\sqrt{1+y^{2}}v_{z}$, we may rewrite the previous system in terms of the new variables as
        \begin{equation*}
            \begin{split}
                \dot{x} & = \bar{v}_{x} \quad \quad \dot{v_x} = -\frac{y}{\sqrt{1+y^{2}}} \frac{\bar{v}_{x}\bar{v}_{y}}{1+y^{2}} \\
                \dot{y} & = \bar{v}_{y}\quad \quad \dot{v_y} = 0 \\
                \dot{z} & = \bar{v}_{z} \quad \quad \dot{v_z} = \frac{1}{\sqrt{1+y^{2}}}\frac{\bar{v}_{x}\bar{v}_{y}}{1+y^{2}}.
            \end{split}
        \end{equation*}
        and after differentiating $\bar{v}_{x}$, $\bar{v}_{y}$ and $\bar{v}_{z}$ and restricting to $\D$, we may obtain the equations fully in terms of the new variables
        \begin{equation*}
            \begin{split}
                \dot{x} & = \bar{v}_{x} \quad \quad \dot{\bar{v}}_x = 0 \\
                \dot{y} & = \bar{v}_{y}\quad \quad \dot{\bar{v}}_y = \frac{y \bar{v}_{y}^{2}}{1+y^{2}} \\
                \dot{z} & = \bar{v}_{z} \quad \quad \dot{\bar{v}}_{z} = \bar{v}_{x}\bar{v}_{y}.
            \end{split}
        \end{equation*}
        This system of equations is equivalent to system \eqref{h:equations}. Hence, the geodesics of $h$ coincide with the trajectories of $\sqrt{1+y^{2}}\Gamma_{(g,\D)}$ and, thus, they are a reparametrization of nonholonomic trajectories.
        
        If a potential function $V$ independent of $z$ is also present, then the nonholonomic trajectories associated with $L=L_{g}-V$ are reparametrizations of the mechanical trajectories associated with the Lagrangian $L_{h}-V$ and initial velocity on $\mathcal{D}$.
    \end{example}

    \begin{remark}\label{remark:lagrangian:side}
        In this example, we have observed that the tangent map of the transformation $\bar{\psi}:\D \to \D$ sending $v_{q}\mapsto e^{-\phi(\pi(q))} v_{q}$ is mapping the vector field $e^{-\phi(\pi(q))}\Gamma_{(g,\D)}$, where $\Gamma_{(g,\D)}$ is the vector field determining nonholonomic trajectories, into the geodesic vector field $\Gamma_{h}|_{\D}$ restricted to the distribution $\D$.
    \end{remark}

    \begin{example}\label{example3}[The Veselova problem]
        Consider a nonholonomic kinetic system in the manifold of three-dimensional rotation matrices $Q=SO(3)$. The Riemannian metric is the left-invariant metric corresponding to the inner-product $(\Omega_{1},\Omega_{2})_{\mathbb{I}} = \Omega_{1}^{T} \mathbb{I} \Omega_{2}$ on $\mathfrak{so}(3)$, where $\Omega_{i}$ are identified with vectors in $\R^{3}$ through the hat map $\hat{(\cdot)}:\R^{3}\to\mathfrak{so}(3)$ defined by
        $$\hat{\Omega}=\begin{bmatrix}
            0 & -\Omega_{3} & \Omega_{2} \\
            \Omega_{3} & 0 & -\Omega_{1} \\
            -\Omega_{2} & \Omega_{1} & 0
        \end{bmatrix},$$
        and $\mathbb{I}:\R^{3}\to \R^{3}$ is the inertia tensor. An important fact about the inertia tensor that we will use later is that there exists a diagonal matrix $A = \text{diag}(a_{1}, a_{2}, a_{3})$ with positive entries such that
        $$\mathbb{I}(u\times v) = (A u) \times (A v).$$
        For example, if $\mathbb{I}=\text{diag}(I_{1}, I_{2}, I_{3})$ with non-vanishing positive diagonal entries then $A=\text{diag}\left(\sqrt{\frac{I_{2}I_{3}}{I_{1}}}, \sqrt{\frac{I_{1}I_{3}}{I_{2}}}, \sqrt{\frac{I_{1}I_{2}}{I_{3}}}\right).$
        
        Now, let $\{e_{1}, e_{2}, e_{3}\}$ be the canonical basis of $\R^{3}$ and consider the right-invariant distribution whose value at the identity is the subspace
        $$\mathfrak{d} = \{\omega\in \mathfrak{so}(3)\ | \ (\omega, e_{3})=0\}.$$
        Nonholonomic systems on Lie groups where the Lagrangian function is left-invariant and the constraint distribution is right-invariant are usually called LR nonholonomic systems. Despite the apparent asymmetry of the system, there is a symmetry under the left-action of the Lie group
        $$SO(2)=\{h\in SO(3) \ | \ h^{-1}e_{3}=e_{3}\},$$
        that is the group of rotations around the axis $e_{3}$. Under this action, the system is a Chaplygin system whose reduced space is $\bar{Q}=\mathbb{S}^{2}$ with projection $\pi:Q\to \bar{Q}$ given by $\pi(g)=g^{-1}e_{3}$. Here, $$\mathbb{S}^{2}=\{\gamma = (x,y,z)\in \R^{3} \ | \ x^{2}+y^{2}+z^{2}=1\}.$$
        The right-invariant distribution generated by the subspace $\mathfrak{d}$ of the Lie algebra happens to be invariant under the action of $SO(2)$ and thus it is an horizontal distribution for $\pi$, while the vertical distribution $V\pi=\ker T\pi$ is the left invariant distribution given by
        $$V_{g}\pi=\{X_{g}\in T_{g}SO(3) \ | \ (g^{-1}X_{g}) \cdot e_{3} = 0\}.$$
        It is better described by its value at the identity
        $$V_{e}\pi=\{\Omega\in \R^{3} \ | \ \Omega \times e_{3}=0\} = \text{span}\{e_{3}\}.$$
        In addition, again following \cite{GaMa}, this system is also a $\phi$-simple system with $\phi(\gamma)=-\frac{1}{2}\ln(A\gamma, \gamma)$ and $(\cdot, \cdot)$ being the euclidean inner product on $\R^{3}$.

        The reduced Riemannian metric $\bar{g}$ on $\bar{Q}$ is given in stereographic coordinates $(x,y)$ on the northern hemisphere of the sphere by the matrix
        $$\bar{g}=\begin{bmatrix}
            a & b \\
            b & c
        \end{bmatrix}$$
        with
        \begin{equation*}
            \begin{split}
                a & = (A\gamma, \gamma) \left( a_{1} + \frac{a_{3}x^{2}}{\sqrt{1-x^2 -y^2}}\right) - x^{2}(a_{3}-a_{1})^{2} \\
                b & = (A\gamma, \gamma) \left( \frac{a_{3}x y}{\sqrt{1-x^2 -y^2}}\right) - x y(a_{3}-a_{1})(a_{3}-a_{2}) \\
                c & = (A\gamma, \gamma) \left( a_{2} + \frac{a_{3}y^{2}}{\sqrt{1-x^2 -y^2}}\right) - y^{2}(a_{3}-a_{2})^{2}
            \end{split}
        \end{equation*}
        as obtained in \cite{GaMa}. Thus, the reduced trajectories are reparametrizations of the geodesics of the Riemannian metric
        $$g_{can}=\begin{bmatrix}
            (A\gamma,\gamma)^{-1}a & (A\gamma,\gamma)^{-1}b \\
            (A\gamma,\gamma)^{-1}b & (A\gamma,\gamma)^{-1}c
        \end{bmatrix}$$

    Let us define the principal Riemannian metric $h$ on $SO(3)$, starting by prescribing its values at the identity $e\in SO(3)$. For all $X, Y \in \mathfrak{X}(\bar{Q})$ and $Z, W\in \Gamma(V)$ we have that
    \begin{equation*}
        \begin{split}
             & h_{e}(X_{e}^{h}, Y_{e}^{h}) = (g_{can})_{\pi(e)}(X,Y)=(g_{can})_{e_{3}}(X,Y) \\
             & h_{e}(X_{e}^{h}, Z_{e}) = 0 \\
             & h_{e}(Z, W) = (Z_{e},W_{e})_{\mathbb{I}}
        \end{split}
    \end{equation*}
    The left trivialization of the horizontal lifts of the coordinate vector fields on $\bar{Q}$ are obtained in \cite{GaMa} to be
    $$\left(\frac{\partial}{\partial x}\right)_{g}^{h}=\pi(g)\times \frac{\partial}{\partial x}, \text{ and } \left(\frac{\partial}{\partial y}\right)_{g}^{h}=\pi(g)\times \frac{\partial}{\partial y},$$
    where the coordinate vector fields are given in terms of the canonical basis of $\R^{3}$ by the expressions
    $$\frac{\partial}{\partial x} = e_{1} - \frac{x}{\sqrt{1-x^2-y^2}} e_3, \text{ and } \frac{\partial}{\partial y}=e_{2} - \frac{y}{\sqrt{1-x^2-y^2}} e_3.$$
    Therefore, at the identity we have that
    $$\left(\frac{\partial}{\partial x}\right)_{e}^{h}= e_{2}, \text{ and } \left(\frac{\partial}{\partial y}\right)_{e}^{h}= - e_{1},$$
    which implies that $h_{e}(e_{1}, e_{2})= -(g_{can})_{e_{3}}(\frac{\partial}{\partial y},\frac{\partial}{\partial x})=0$, $h_{e}(e_{1}, e_{1})=(g_{can})_{e_{3}}(\frac{\partial}{\partial y},\frac{\partial}{\partial y})=a_{2}$ and $h_{e}(e_{2}, e_{2})= (g_{can})_{e_{3}}(\frac{\partial}{\partial x},\frac{\partial}{\partial x})=a_{1}$. In addition, $V_{e}\pi =\text{span}\{ e_{3} \}$ and so $h_{e}(e_{3}, e_{3})=e_{3}^{T}\mathbb{I}e_{3} = \mathbb{I}_{33}$. Therefore,
    $$h_{e}(\Omega_{1}, \Omega_{2})= \Omega_{1}^{T}\begin{bmatrix}
        a_{2} & 0 & 0 \\
        0 & a_{1} & 0 \\
        0 & 0 & \mathbb{I}_{33}
    \end{bmatrix} \Omega_{2}, \quad \Omega_{1,2} \in \mathfrak{so}(3).$$
    
    However, the principal Riemannian metric $h$ might not be left-invariant. In fact, if $\bar{X}, \bar{Y} \in \mathfrak{X}(\bar{Q})$, then $h_{g}(\bar{X}_{g}^{h}, \bar{Y}_{g}^{h})=g_{can, \pi(g)}(\bar{X},\bar{Y})= e^{2\phi(\pi(g))} \bar{g}(\bar{X},\bar{Y})$. This implies that the Riemannian metric $h$ will in general depend on the point $g\in SO(3)$ and will not be left-invariant. It can be easily checked that $h_{g}(\left(\frac{\partial}{\partial x}\right)_{g}^{h}, \left(\frac{\partial}{\partial x}\right)_{g}^{h})$ does not coincide with $h_{e}(\pi(g)\times \frac{\partial}{\partial x}, \pi(g)\times \frac{\partial}{\partial x})$ for an arbitrary matrix $A$ and $g\in SO(3)$.

    
    Therefore, the expression of $h$ can be quite different at different points of $SO(3)$. In any case, the principal Riemannian metric $h$ in an arbitrary point $g\in SO(3)$  with respect to the reference frame $\{X_{1} = \left(\frac{\partial}{\partial x}\right)_{g}^{h}, X_{2} = \left(\frac{\partial}{\partial y}\right)_{g}^{h}\}$ generating $\D$ and $X_{3}=g\cdot e_{3}$, generating $V$, is represented by the matrix
    $$h(X_{i}, X_{j}) = \begin{bmatrix}
        (A\gamma,\gamma)^{-1}a & (A\gamma,\gamma)^{-1}b & 0 \\
        (A\gamma,\gamma)^{-1}b & (A\gamma,\gamma)^{-1}c & 0 \\
        0 & 0 & \mathbb{I}_{33}
    \end{bmatrix}$$

    A more detailed study of the nature of the Riemannian metric $h$ will be postponed to a future paper.
    \end{example}

    \begin{remark}\label{multi:dimensional:Veselova}
        A multi-dimensional version of the Veselova problem with configuration space the Lie group $SO(n)$ was considered in \cite{GaMa}. This system is also a $\phi$-simple Chaplygin. So, their trajectories are reparametrizations of the trajectories of an unconstrained mechanical Lagrangian system on $SO(n)$. 
    \end{remark}
    
	\section{Conclusions and future work}
	We prove that unreduced $\phi$-simple Chaplygin mechanical systems are mechanically Hamiltonizable up to reparametrization. More precisely, we prove that it is possible to modify the Riemannian metric associated to the Kinetic energy of the system to produce a new Riemannian metric. Then, the trajectories of the unconstrained mechanical Lagrangian system induced by the new metric and the same potential energy are reparametrizations of the original nonholonomic trajectories. Our method is constructive and it is easily applied to several interesting examples.
    
	The previous results rise several interesting questions:
	\begin{enumerate}	
		\item If the symmetries are not just of kinematic type (as in Chaplygin systems) but there are, for instance, horizontal symmetries, can we still do something?
		
		\item There is evidence in the literature that, under the same assumptions, there are many more Lagrangian functions w.r.t. which the nonholonomic trajectories are the solution of Euler-Lagrange equations. Namely, we can also find non-positive-definite semi-Riemannian metrics whose geodesics are nonholonomic trajectories. So, what is the more general picture?

        \item Take advantage of the modified Riemannian metrics to produce numerical integrators for nonholonomic systems (see \cite{CM2001}).

        \item It is interesting to note that the situation described in Remark \ref{remark:lagrangian:side} can be generalized to all of our examples. In fact, if $(Q,g, V, G, \D)$ is a $\phi$-simple Chaplygin system with $G$-invariant potential $V$, $h$ is the principal Riemannian metric constructed in Theorem \ref{potential:thm} and $\tau_{Q}:TQ\to Q$ is the canonical tangent projection on $Q$, then the tangent map of the transformation $\bar{\psi}:\D \to \D$ sending $v_{q}\mapsto e^{-\phi(\pi(q))} v_{q}$ will map the vector field $e^{-(\phi\circ\pi\circ (\tau_{Q})|_{\D})}\Gamma_{(g,V,\D)}$, where $\Gamma_{(g,V,\D)}$ is the vector field determining the nonholonomic trajectories of this system, into the mechanical vector field $\Gamma_{(h,V)}|_{\D}$ restricted to the distribution $\D$.

        We will give a sketch of the proof of this result. Denote by $(T\pi)|_{\D}: \D\to T\bar{Q}$ the restriction to $\D$ of the tangent map of $\pi:Q \to \bar{Q}$ and by $\Gamma_{nh}\in \mathfrak{X}(T\bar{Q})$ the $(T\pi)|_{\D}$-projection of the vector field $\Gamma_{(g,V,\D)}$ (note that $\Gamma_{(g,V,\D)}$ is $G$-invariant and, thus, it is $(T\pi)|_{\D}$-projectable). If $g_{can}$ is the Riemannian metric on $\bar{Q}$ given by $g_{can}=e^{2\phi}\bar{g}$, $\bar{V}\in C^{\infty}(\bar{Q})$ is such that $V=\bar{V}\circ \pi$ and $\Gamma_{(g_{can},\bar{V})}$ is the dynamical vector field in $T\bar{Q}$ associated with the mechanical Lagrangian system $L_{(g_{can},\bar{V})}$ then, using Corollary \ref{reparametrization:mechanical:system}, we deduce that these vector fields are related by the vector bundle isomorphism
        $$\tilde{\bar{\psi}}=\sharp_{g_{can}}\circ \psi \circ \flat_{\bar{g}}:T\bar{Q}\to T^{*}\bar{Q}\to T^{*}\bar{Q}\to T\bar{Q}.$$
        Here, $\flat_{\bar{g}}:T\bar{Q} \to T^{*}\bar{Q}$, $\sharp_{g_{can}}:T^{*}\bar{Q}\to T\bar{Q}$ are the musical isomorphisms induced by the Riemannian metrics $\bar{g}$ and $g_{can}$, respectively, and $\psi:T^{*}\bar{Q}\to T^{*} \bar{Q}$ is the isomorphism in Proposition \ref{vector:bundle:isomorphism}. Now, a direct computation proves that
        $$\tilde{\bar{\psi}}(v_{\bar{q}})=e^{-\phi(\bar{q})}v_{\bar{q}}, \text{ for } v_{\bar{q}}\in T_{\bar{q}}\bar{Q}.$$
        On the other hand, considering $(\tau_{Q})|_{\D}$ the restriction of $\tau_{Q}$ to $\D$, the trajectories in $Q$ of the vector fields $e^{-(\phi\circ\pi\circ(\tau_{Q})|_{\D})}\Gamma_{(g,V,\D)}$ $\in\mathfrak{X}(\D)$ and $\Gamma_{(h,V)}|_{\D}$ are the horizontal lifts, with respect to the Riemannian submersion $\pi: (Q,h)\to (\bar{Q}, g_{can})$, of the trajectories in $\bar{Q}$ of the vector fields $e^{-(\phi\circ\tau_{\bar{Q}})}\Gamma_{nh}\in\mathfrak{X}(T\bar{Q})$ and $\Gamma_{(g_{can}, \bar{V})}\in \mathfrak{X}(T\bar{Q})$, respectively, where $\tau_{\bar{Q}}:T\bar{Q}\to \bar{Q}$ is the canonical tangent projection on $\bar{Q}$. Therefore, using that the following diagram
        \begin{figure}[htb!]
            \centering
            \begin{tikzcd}
                \mathcal{D} \arrow[d, "(T\pi)|_{\mathcal{D}}"'] \arrow[r, "\bar{\psi}"] & \mathcal{D} \arrow[d, "(T\pi)|_{\mathcal{D}}"] \\
                T\bar{Q} \arrow[r, "\tilde{\bar{\psi}}"]                                & T\bar{Q}                                      
            \end{tikzcd}
        \end{figure}
        is commutative, we conclude that the vector fields $e^{-(\phi\circ \pi \circ (\tau_{Q})|_{\D})}\Gamma_{(g,V,\D)}\in \mathfrak{X}(\D)$ and $\Gamma_{(h,V)}\in \mathfrak{X}(\D)$ are $\bar{\psi}$-related.
        \end{enumerate}
        
\section*{Acknowledgements}

The authors acknowledge financial support from Grants PID2022-137909-NB-C21, PID2022-137909-NB-C22 and RED2022-134301-TD funded by the Spanish Ministry of Science and Innovation. The authors are also thankful to L. Garc\'ia-Naranjo for some useful comments on a first version of this paper.

\printbibliography
	
\end{document}